\documentclass{amsart}
\usepackage{amssymb}
\usepackage{tikz}
\usepackage{verbatim}
\usepackage{hyperref}
\usepackage[english]{babel}
\usepackage{wasysym}
\usepackage[shortlabels]{enumitem}
\usepackage[colorinlistoftodos, textsize=small]{todonotes}
\usepackage{color}
\usepackage{soul}
\usepackage{mathtools}
\usepackage{mathrsfs}
\usepackage{bm}
\usepackage{graphicx}
\usepackage{float}

\newtheorem*{cla}{Claim}
\newtheorem*{cla1}{Claim 1}
\newtheorem*{cla2}{Claim 2}
\newtheorem*{cla3}{Claim 3}
\newtheorem*{cla4}{Claim 4}

\newenvironment{claimproof}{\paragraph{\em Proof of claim.}}{\hfill\scalebox{0.7}{$\blacksquare$}}

\newtheorem{theorem}{Theorem}[section]
\newtheorem{remark}[theorem]{Remark}
\newtheorem{example}[theorem]{Example}
\theoremstyle{definition}

\theoremstyle{lemma}
\newtheorem{lemma}[theorem]{Lemma}
\theoremstyle{corollary}
\newtheorem{corollary}[theorem]{Corollary}
\theoremstyle{proposition}
\newtheorem{proposition}[theorem]{Proposition}
\theoremstyle{notation}

\theoremstyle{problem}

\theoremstyle{conjecture}

\numberwithin{equation}{section}

\newcommand{\alg}[1]{\mathbf{#1}}
\newcommand{\var}[1]{\mathcal{#1}}

\newcommand{\pol}{\operatorname{Pol}}
\newcommand{\clo}{\operatorname{Clo}}

\newcommand{\id}{\operatorname{id}}
\newcommand{\Id}{\operatorname{Id}}

\newcommand{\con}{\operatorname{Con}}

\newcommand\pow[1]{\mathcal{P}(#1)}

\DeclareMathAlphabet\mathbfcal{OMS}{cmsy}{b}{n}

\newcommand{\LL}{\mathbfcal{L}}
\newcommand{\ultra}{\mathcal{U}}
\newcommand{\vltra}{\mathcal{\bar U}}

\newcommand{\mA}{\mathbb{A}}
\newcommand{\mB}{\mathbb{B}}
\newcommand{\mC}{\mathbb{C}}
\newcommand{\mD}{\mathbb{D}}
\newcommand{\mE}{\mathbb{E}}
\newcommand{\mF}{\mathbb{F}}
\newcommand{\mG}{\mathbb{G}}
\newcommand{\mH}{\mathbb{H}}
\newcommand{\mK}{\mathbb{K}}
\newcommand{\mL}{\mathbb{L}}
\newcommand{\mP}{\mathbb{P}}
\newcommand{\mQ}{\mathbb{Q}}
\newcommand{\mU}{\mathbb{U}}

\begin{document}

\title{Taylor is prime}

\address{Bolyai Institute, Univ. of Szeged, Szeged, Aradi V\'{e}rtan\'{u}k tere 1, HUNGARY 6720}

\author{Bertalan Bodor} 
\email{bodor@server.math.u-szeged.hu} 

\author{Gerg\H o Gyenizse} 
\email{gergogyenizse@gmail.com}    
           
\author{Mikl\'os Mar\'oti}
\email{mmaroti@math.u-szeged.hu} 
            
\author{L\'aszl\'o Z\'adori}
\email{zadori@math.u-szeged.hu}

\thanks{\thanks{The research of authors was supported by the NKFIH grant K138892 and Project no TKP2021-NVA-09 where the latter has been financed by the Ministry of Culture and Innovation of Hungary from the National Research, Development and Innovation Fund.}}


\begin{abstract}
We study the Taylor varieties and obtain new characterizations of them via compatible reflexive digraphs. Based on our findings, we prove that in the lattice of interpretability types of varieties, the filter of the types of all Taylor varieties is prime. 
\end{abstract}

\maketitle

\section{Introduction}

We require some basic concepts to introduce the topic of our investigations in this paper. Let $A$ be a set. An $n$-ary operation $f$ on $A$ is {\em idempotent} if $f(a,\dots,a)=a$ for all $a\in A$.
 A {\em clone on} $A$ is a set  of operations on $A$ that contains all projection operations and is closed under composition.
Let $\alg A$ be an algebra. The {\em clone of} $\alg A$ denoted by $\clo (\alg A)$ is the clone of term operations of $\alg A$. The {\em clone of a variety} $\var V$ is the clone of the free algebra generated by countably many free generators. An {\em algebra (a variety)  is idempotent} if all operations in its clone are idempotent. 

A {\em clone homomorphism} from a clone $C$ to a clone $D$ is a map from $C$ to $D$ that maps each projection of $C$ to the corresponding projection of $D$ and commutes with composition. A {\em variety $\var V$ interprets in a variety $\var W$} if there is a clone homomorphism from the clone of $\var V$ to the clone of $\var W$. 

As easily seen, interpretability is a quasiorder on the class of varieties. The blocks of this quasiorder are called  {\em interpretability types}. In \cite{GT} Garcia and Taylor introduced the {\em lattice $\LL$ of interpretability types of varieties} that is obtained by taking the quotient of the class of varieties quasiordered by  interpretability and the related equivalence. The join in  $\LL$ is described as follows. Let $\mathcal{V}_1$  and $\mathcal{V}_2$ be two varieties of disjoint signatures. Let $\mathcal{V}_1$  and $\mathcal{V}_2$  be defined by the sets $\Sigma_1$ and $\Sigma_2$ of identities, respectively. {\em Their join} $\mathcal{V}_1\vee\mathcal{V}_2$ is the variety defined by $\Sigma_1\cup\Sigma_2$. The so defined join is compatible with the interpretability relation of varieties, and naturally yields the definition of the join operation in $\LL$.

A {\em strong Maltsev condition} is a finite set of identities in the language of a variety. A {\em Maltsev condition} is a sequence $M_n,\ n\geq 1,$ of strong  Maltsev conditions where the variety defined by $M_{n+1}$ interprets in the variety defined by $M_n$ for each $n\geq 1$. We say that a {\em variety $\var V$ admits a Maltsev condition  $M_n,\ n\geq 1,$ (or the terms that occur in a Maltsev condition  $M_n,\ n\geq 1$)}  if there exists an $n$ such that the variety defined by $M_n$ interprets in $\var V$.  By a {\em filter}, we mean an upwardly closed sublattice of a lattice. Clearly, the types of the varieties that admit a particular Maltsev condition form a filter in $\LL$.  Such filters of $\LL$ are called {\em Maltsev filters}. A filter $F$ of a lattice is called {\em prime} if $F$ is a proper subset of the lattice and for any two elements $a,b\not\in F$,  we have $a\vee b\not\in F$.
To gain information on the structural properties of $\LL$, it is useful to determine which of the well-known classical Maltsev filters are prime filters in $\LL$. In case a Maltsev filter is (not) prime, we just say sometimes that the related Maltsev condition is (not) prime. 

In \cite{GT}, primeness of some well-known Maltsev filters of $\LL$ was decided. For example, in \cite{GT}, it was proved that  the filter containing the largest element of $\LL$, given by the Maltsev condition $\{x=y\}$, the filter determined by the Maltsev condition $\{s(x,y)=s(y,x)\}$, and the filter determined by congruence regularity are all prime in $\LL$. However, the filter determined by congruence distributivity is not prime in $\LL$. The primeness of some Maltsev filters, such as the filter of the types of congruence permutable varieties and the filter of the types of congruence modular varieties, were stated as open problems in \cite{GT}. Not long ago, the last  three authors of this paper proved that congruence permutability is indeed a prime Maltsev condition, see \cite{GMZ2}. The problem on congruence modular varieties is still open. 

By extending the results of Hobby and McKenzie in \cite{HM}, Kearnes and Kiss elaborated a hierarchy of varieties based on certain Maltsev conditions in \cite{KK}. It is natural to look at the problem of primeness for the Maltsev filters in $\LL$ that were characterized in \cite{KK}. The largest Maltsev filter described in \cite{KK}, introduced by Taylor in \cite{T}, and studied in \cite{HM} is the class of the interpretability types of Taylor varieties ({\em Taylor interpretability types} for short). In the present paper, we investigate the class of Taylor varieties, and prove that the filter of Taylor interpretability types is  prime in $\LL$. 

In Figure \ref{fig:hierarch}, we depicted the Maltsev filters of $\LL$ that were used to establish a certain hierarchy of varieties in \cite{KK}. A node is painted black if it symbolizes a prime Maltsev filter, is left empty if it symbolizes a non-prime filter, and is painted gray if its status is unknown at present. The filters in the figure are depicted accordingly to reverse containment. This may be unusual but we want to keep the direction coming from the interpretability ordering. It is well known that the filters in the figure are pairwise different. 

\begin{figure}[H] 
\centering
\includegraphics[scale=1]{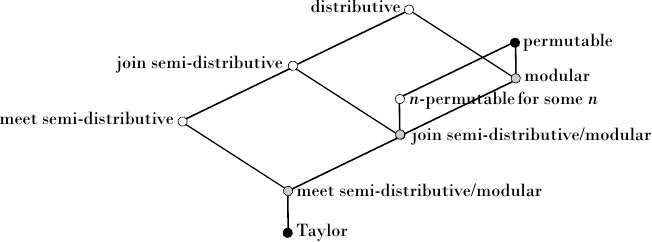}
\caption{Maltsev filters that establish a hierarchy of varieties.} 
\label{fig:hierarch}
\end{figure}

In a short while below, we define Taylor varieties. Then the meaning of the filter of Taylor interpretability types that is the main object of our investigations in the present paper becomes clear. For the definitions of the other Maltsev filters in Figure \ref{fig:hierarch}, the reader should consult  \cite{KK}.  Note that monograph \cite{KK} contains a wide variety of characterizations of the filter of Taylor interpretability types and most of the other filters in the figure. 

We remark that the non-primeness of $n$-permutability for some $n$ was proved by the last three authors of the present paper in \cite{GMZ1}. To verify the status of distributivity, join and meet semi-distributivity in Figure \ref{fig:hierarch}, we render a  proposition. 

We need the following definitions. A lattice is {\em  meet semi-distributive} if it satisfies the implication $x\wedge y=x\wedge z\Rightarrow x\wedge y=x\wedge (y\vee z)$. The {\em join semi-distributivity} of a lattice is defined dually.
A variety $\var V$ is {\em congruence meet semi-distributive (congruence join semi-distributive, congruence distributive)} if the congruence lattice of any algebra in $\var V$ has the same property. A ternary term $t$ in the language of a variety $\var V$ is called a {\em a majority term for $\var V$} if $t$ satisfies the identities $t(y,x,x)=t(x,y,x)=t(x,x,y)=x$ in $\var V$.

\begin{proposition}
    The following Maltsev  conditions  are not prime:
    \begin{enumerate}
        \item congruence meet semi-distributivity,
        \item congruence join semi-distributivity,
        \item congruence distributivity,
        \item admittance of a majority term. 
    \end{enumerate}
\end{proposition}
\begin{proof} As well known, see for example \cite{KK}, any variety that satisfies a Maltsev condition on the above list also satisfies the ones prior to it. So to prove the proposition, it suffices to present two varieties that are not meet semi-distributive such that their join admits a majority term. 

Let $\var V_1$ be the variety  defined by the identities $$m(x,y,y)=m(y,x,y)=m(y,y,x)=x$$ for a single ternary operation symbol $m$, and $\var V_2$ the variety defined by the identities $$s(x,x)=x\text{ and }s(x,y)=s(y,x)$$ for a single  binary operation symbol $s$. 

We take the algebras $$\alg A_1 =(Z_2;\ x+y+z)\in \var V_1 \text{ and }\alg A_2= (Z_3;\ 2x+2y)\in \var V_2 $$ and determine the congruence lattices of the algebras $\alg A_i^2\in\var V_i$, $1\leq i\leq 2$. Notice that these squares have the same polynomials as the groups $\alg Z_{i+1}^2=(Z_{i+1};\ +)^2$, $1\leq i\leq 2$, respectively. So $\con(\alg A_i^2)\cong \con(\alg Z_{i+1}^2) \cong \alg M_{i+2}$, $1\leq i\leq 2$, where $\alg M_j$ denotes the $(j+2)$-element lattice with $j$ atoms. Since these congruence lattices are not  meet semi-distributive, the varieties $\var V_1$ and $\var V_2$ are not meet semi-distributive either, but their join  has a majority term  of the form $m(s(x,y),s(x,z),s(y,z))$. 
\end{proof}

We remark that the question of primeness of an idempotent Maltsev condition makes sense restricted to the sublattice $\LL_{\Id}$ of interpretability types of idempotent varieties in $\LL$. While the primeness of an idempotent Maltsev condition in $\LL$ implies the primeness of the same idempotent Maltsev condition in $\LL_{\Id}$, the converse does not hold in general. In this respect, we mention that in the proof of the above proposition, only idempotent varieties were used, so the proposition holds if we consider  interpretability in $\LL_{\Id}$, that is, majority, distributivity, join semi-distributivity, and meet semi-distributivity are not prime Maltsev conditions in $\LL_{\Id}$ as well. The primeness of  the filter of Taylor interpretability types in $\LL_{\Id}$ follows from Taylor's work in \cite{T}. We shall give an argument below.  The primeness of join semi-distributive over modular in $\LL_{\Id}$  immediately follows from Lemma 9.5 in \cite{HM}. It was proved in \cite{VW} that $n$-permutability for some $n$ is prime in $\LL_{\Id}$.  The primeness of  modularity relative to $\LL_{\Id}$ was obtained in \cite{OP}. The primeness of  permutability in $\LL_{\Id}$ was first proved in \cite{KT}. We are not aware  of any result related to the primeness of meet-semidistributive over modular in $\LL_{\Id}$ or in $\LL$. 

Let $\mathcal{S}\mathcal{E}\mathcal{T}$ denote the variety of algebras with no basic operations. We note that $\mathcal{S}\mathcal{E}\mathcal{T}$ interprets in every variety, and hence its interpretability type is the smallest element of $\LL$. A {\em Maltsev condition is trivial} if the variety of $\mathcal{S}\mathcal{E}\mathcal{T}$ admits it. A {\em Taylor variety} is a variety which admits a non-trivial idempotent (strong) Maltsev condition. Taylor varieties play a crucial role in the theory of  general algebras \cite{KK}, in tame congruence theory \cite{HM}, and in the theory of constraint satisfaction problems \cite{LT}. In all of these areas, the Taylor varieties are considered the well-behaved varieties from many points of view. Meanwhile, the non-Taylor varieties are looked on as the black sheep of the family with wild properties to be expected for them.

 An identity is {\em linear} if it has at most one occurrence of a function symbol on each side of the identity. A {Taylor term} is an $n$-ary idempotent term $t$ for some $n$ such that for each $1\leq i\leq n$, $t$ satisfies a linear identity  $$t(\dots,x,\dots)=t(\dots,y,\dots)$$ in two different variables $x$ and $y$, where the variables $x$ and $y$ displayed are in the $i$-th positions respectively on both sides, and the occurrences of the variables not displayed are  arbitrarily  set $x$ or $y$. 

The {\em full idempotent reduct of a variety} $\var V$ is a variety whose signature is the set of idempotent terms of $\var V$, and whose identities are those satisfied by the idempotent terms for $\var V$. A {\em group $\alg G$ is compatible in a variety $\var V$} if there is an algebra $\alg A$ in $\var V$ such that the underlying sets of $\alg A$ and $\alg G$ are the same and the operations of $\alg A$ commute with the group operation of $\alg G$. The following characterization of Taylor vaieties is due to Taylor, cf. Theorems 5.1, 5.2, and 5.3 in \cite{T}.

\begin{theorem}\label{taylor} 
Let $\var V$ be a variety, and $\var V_{\Id}$ the full idempotent reduct of $\var V$. Then the following are equivalent.
\begin{enumerate}
\item $\var V$ is a Taylor variety.
\item Every compatible group is commutative in $\var V_{\Id}$. 
\item $\var V_{\Id}$ does not interpret in $\mathcal{S}\mathcal{E}\mathcal{T}.$
\item $\var V$ admits a Taylor term.
\end{enumerate}
\end{theorem}

In Theorem 1.1 of \cite{O}, Ol\v s\'ak gave a characterization of Taylor varieties by some non-trivial idempotent strong Maltsev condition, and hence their interpretability types form a principal filter in $\LL$. An {\em Ol\v s\'ak term} is a 6-ary idempotent term $t$ that satisfies the following identities $$t(x,y,y,y,x,x)=t(y,x,y,x,y,x)=t(y,y,x,x,x,y).$$

\begin{theorem} 
A variety $\var V$ is a Taylor variety if and only if $\var V$  admits an Ol\v s\'ak term.
\end{theorem}

Further equivalent conditions characterizing Taylor varieties can be found in the monograph \cite{KK} of Kearnes and Kiss. Notice that by Theorem \ref{taylor}, if $\var V_1$ and $\var V_2$ are idempotent non-Taylor varieties, then they both interpret in $\mathcal{S}\mathcal{E}\mathcal{T}$. So their join also interprets in  $\mathcal{S}\mathcal{E}\mathcal{T}$. Hence, the filter of  Taylor interpretability types is prime in $\LL_{\Id}$. 

The main result of this paper states that we can drop idempotency from this statement, that is, 
the filter of Taylor interpretability types is prime in $\LL$. For the proof, we shall derive a novel characterization  of Taylor varieties in Section 4.  We present some  preliminary results on disjoint unions of powers of the reflexive directed triangle digraph in Sections 2 and 3. We note at this point that the second part of Section 2 following Corollary \ref{cloneofc}, and the  entire Section 3 cover materials that are not used in the proof of the main theorem. These parts of the paper contain motivating examples and results that led to our characterization of Taylor varieties in Section 4. So the reader may choose to skip proofs in these parts of the paper when first reading it without losing essentials to the content of Section~4.

\section{Some basic facts on powers of $\mC$}

Throughout this paper, we use blackboard bold capital letters to denote relational structures, in particular digraphs, and use bold capital letters to denote algebras.  Usually, we use the same capital italics to denote the base sets of these structures.
We denote by $\mC$ the 3-cycle graph together with all loops. We use the following notational convention: the vertices of $\mC$ are labeled $0,1,2$ and the edges of $\mC$ are $(0,1),(1,2),(2,0),(0,0),(1,1),(2,2)$, see Figure \ref{fig:triangle}.
	
\begin{figure}[H] 
\centering
\includegraphics[scale=0.7]{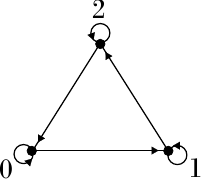}
\caption{The digraph $\mC$.} 
\label{fig:triangle}
\end{figure}
	
	Let $\mG$ and $\mH$  be digraphs. An edge preserving map $\varphi\colon G\to H$ is called a {\em homomorphism from $\mG$ to $\mH$.}  A homomorphism  $\varphi\colon \mG \to \mH$ is a {\em retraction}  if there exists a homomorphism $\varepsilon\colon \mH \to \mG$ where $\varphi\circ\varepsilon=\id_H$. A digraph $\mH$ is a {\em retract} of $\mG$ if there is a retraction from $\mG$ onto $\mH$. We say that $\mH$ is a {\em subdigraph} of $\mG$ if every vertex of $\mH$ is a vertex of $\mG$ and every edge of $\mH$ is an edge of $\mG$. We use $\mH\subseteq\mG$ to denote that
$\mH$ is a subdigraph of $\mG$. We call $\mH$ {\em a spanned subdigraph of }$\mG$, if $\mH\subseteq\mG$ and every edge of $\mG$ whose end vertices are in $H$ is also an edge of $\mH$. A {\em digraph is connected} if the reflexive, symmetric, transitive closure of its edge relation is the full relation. The {\em (connected) components} of a digraph are the maximal connected spanned subdigraphs.
	
Let $\mH=(H,\rightarrow)$ be a digraph and $I$ an arbitrary set. Then we write $\mH^I$ for the digraph whose vertices are the elements of $H^I$, i.e., all functions from $I$ to $H$, and whose set of edges is $\{(u,v)\in (H^I)^2: \forall i\in I,\ u(i)\rightarrow v(i)\}$. One can think of $I$ as a set of coordinates, and then $u(i)$ denotes the $i$-th entry of some vertex $u\in H^I$. For $J\subseteq I$ we write $\pi_J$ for the projection map from $H^I$ to $H^J$, i.e., $\pi_J(a)=a|_J$. This map is clearly a surjective homomorphism from $\mH^I$ to $\mH^J$. A homomorphism from $\mH^I$ to $\mH$ is called an $I$-ary {\em polymorphism} of $\mH$.

We conceive of $C^I$ as the Abelian group $\alg Z_3^I$. So addition and subtraction for the elements of $C^I$ are always understood as in $\alg Z_3^I$. For some subset $J\subseteq I$ we define $\chi_J\in C^I$ by $$\chi_J(j)=1 \text{ if } j\in J, \text{ and } \chi_J(j)=0 \text{ otherwise.}$$ Clearly, every $a\in C^I$ can uniquely be written as $a=\chi_{A}-\chi_{B}$ for some disjoint subsets $A$ and $B$ of $I$.    
 
We call a digraph $\mH$ {\em reflexive} if for any $u\in H$, $u\to u.$  A digraph $\mH$ is {\em antisymmetric} if for any $u,v\in H$, $u\rightarrow v\rightarrow u$ implies $u=v$.

\begin{lemma}\label{image_of_powers}
Let $\mH$ be an antisymmetric digraph such that any non-loop edge of $\mH$ is contained in at most one copy of $\mC$ in $\mH$.
Let $I$ be a finite set, and $f\colon \mC^I\rightarrow \mH$  a homomorphism. Then there exist $J\subseteq I$ and an isomorphism $\iota\colon \mC^J\rightarrow f(\mC^I)$ such that $f=\iota\circ \pi_J$. In particular, the image of $f$ is isomorphic to a finite power of $\mC$.  
\end{lemma}

\begin{proof}
	Let $$J\coloneqq \{j\in I\colon f \text{ depends on its } j\text{-th coordinate}\}.$$ We assume that $J$ is non-empty, for otherwise $f$   is  a constant map and the claim obviously holds. Clearly, then there exists an onto homomorphism $\iota\colon \mC^J\rightarrow f(\mC^I)$ such that $f=\iota\circ \pi_J$. It suffices to prove that $\iota$ is injective and its inverse is also a homomorphism. 

First we prove that $\iota$ is injective. Let us suppose to the contrary that $a$ and $b$ are different elements in $C^J$ such that $\iota(a)=\iota(b)$. Then there are some $A,B\subseteq J$ such that $ A\cap B=\emptyset$ and $b=a+\chi_{A}-\chi_{B}$. Without loss of generality, we may assume that $A$ is non-empty and $j\in A$. Then $b\to a - \chi_{A}\to a $ in $\mC^J$, whence $\iota(b)\to \iota(a - \chi_{A})\to \iota(a) $ in $\mH$. So by $\iota(a)=\iota(b)$ and antisymmetry in $\mH$, $\iota(a - \chi_{A})=\iota(a)$. 

Then
$a - \chi_{A}\to a - \chi_{\{j\}}\to a$ in $\mC^J$ and $\iota(a - \chi_{A})\to \iota(a - \chi_{\{j\}})\to \iota(a)$ in $\mH$. Hence by $\iota(a - \chi_{A})=\iota(a)$ and antisymmetry, $\iota(a - \chi_{\{j\}})= \iota(a).$ Now $ a\to a + \chi_{\{j\}}\to a - \chi_{\{j\}}$ in $\mC^J$ and 
$\iota(a)\to \iota(a + \chi_{\{j\}})\to \iota(a - \chi_{\{j\}})$ in $\mH$. By using antisymmetry again, $\iota(a + \chi_{\{j\}})= \iota(a).$
Thus for $a$ we have that $\iota(a -\chi_{\{j\}})= \iota(a)= \iota(a + \chi_{\{j\}}).$

Clearly, $\iota$ depends on its $j$-th coordinate since $f$ does. So there exists $c\in C^J$ such that $\iota(c - \chi_{\{j\}})$ and $\iota(c)$ are different.  By antisymmetry, $\iota(c - \chi_{\{j\}}),\ \iota(c)$ and $\iota(c + \chi_{\{j\}})$ are pairwise different. So we may assume that $a(j)=c(j)$ for such  $c$. Now there is a sequence in $C^J$ that starts with $a$ and ends with $c$ such that $j$-th components of the members are all equal to $a(j)$ and any two consecutive members differ in a single component. This sequence clearly has two consecutive members $a'$ and $c'$ such that $\iota(a' -\chi_{\{j\}})= \iota(a')= \iota(a' + \chi_{\{j\}})$, and $\iota(c' - \chi_{\{j\}}),\ \iota(c')$ and $\iota(c' + \chi_{\{j\}})$ are pairwise different. So we may assume that there is $k\neq j$ such that  $a(i)=c(i)$ for all $i\in J\setminus \{k\}$, $\iota(a -\chi_{\{j\}})= \iota(a)= \iota(a + \chi_{\{j\}})$, and $\iota(c - \chi_{\{j\}}),\ \iota(c)$ and $\iota(c + \chi_{\{j\}})$ are pairwise different. We may also assume that $a(k)=0$ and $c(k)=1$. 

Let $\mD$ be the subdigraph  spanned by all the tuples $d$ in $\mC^J$  such that $d(i)=a(i)$ for all $i\in J\setminus \{j,k\}$. Let $\varphi=\pi_{\{k,j\}}|_D$. Obviously, $\varphi$ is an isomorphism from $\mD$ to $\mC^2.$ Let $\psi= \iota|_D \varphi^{-1}$. Then  $\psi$ is a homomorphism from $\mC^2$ to $\iota(\mC^J)$ such that by the properties of $a$ and $c$, we have  that $\psi(0,0)=\psi(0,1)=\psi(0,2)$, and $\psi(1,0),\ \psi(1,1),\ \psi(1,2)$ are pairwise different. Then 
\begin{align*}\psi(2,0)\to \psi(0,0)&=\psi(0,1)\to\psi(1,0)\to\psi(2,0), \text{ and }\\ 
\psi(2,0)\to \psi(0,0)&=\psi(0,1)\to\psi(1,2)\to\psi(2,0) \text{ in } \mH.
\end{align*} 
Observe that the set $S$ of the $\psi$-values appearing on the preceding two lines has less than four elements. Indeed, for otherwise the vertices appearing on each line would yield a copy of $\mC$ in $\mH$, and these two different copies of $\mC$ would share the common edge $\psi(2,0)\to \psi(0,0)$, which is impossible. Then by antisymmetry of $\mH$, $S$ is a singleton or a 3-element set. So in both cases, $\psi(1,0)=\psi(1,2)$, a contradiction. Thus $\iota$ has to be injective.
	
We prove that the inverse of $\iota$ is a homomorphism.  Now let $a$ and $b$ be arbitrary elements of $C^J$. Let us assume that $\iota(a)\rightarrow \iota(b)$. Put $b=a+\chi_{A}-\chi_{B}$ for some $A,B\subseteq J, \ A\cap B=\emptyset$. Then let $c=a-\chi_{A}-\chi_{B}$, and $d=a-\chi_{A}$. Clearly, $b\rightarrow c\rightarrow a$ and $b\rightarrow d\rightarrow a$. Therefore, we have $\iota(a)\rightarrow \iota(b)\rightarrow \iota(c)\rightarrow \iota(a)$ and $\iota(a)\rightarrow \iota(b)\rightarrow \iota(d)\rightarrow \iota(a)$. As $\mH$ is antisymmetric and in $\mH$ every non-loop edge is in at most one copy of $\mC$, it follows that $\iota(c)=\iota(d)$, and thus $c=d$. This means exactly that $B=\emptyset$, and thus $a\rightarrow b$.
\end{proof}
 
We note that later in the paper, typical applications of Lemma \ref{image_of_powers} occur when $\mH$ is a subdigraph of some power of $\mC$. For any relational structure $\mQ$, let $\pol(\mQ)$ denote the clone of finitary polymorphisms of $\mQ$. 
By the help of the lemma, $\pol(\mC)$ is easily described.

\begin{corollary} \label{cloneofc}The clone $\pol(\mC)$ consists of all constant operations of $C$ and the essentially automorphism operations of $\mC$. 
\end{corollary}
	
	The following example shows that Lemma~\ref{image_of_powers} fails for infinite powers. Let $I$ be any set, and let $\pow{I}$ denote the lattice of all subsets of $I$. The prime filters of $\pow{I}$ are usually called {\em ultra filters}.

\begin{example}
{\em For any set $I$ and ultra filter $\ultra\subseteq \pow{I}$, the map $$f\colon \mC^I\rightarrow \mC,\ a\mapsto a_\ultra\text{ where }a^{-1}(a_\ultra)\in \ultra$$ is a homomorphism.}
\end{example}

The following lemma gives that in fact, all non-constant homomorphisms from powers of $\mC$ to $\mC$ are of the form as in the preceding example, apart from an automorphism of $\mC$.
	
\begin{lemma}\label{all_hom}
	Let $I$ be an arbitrary set, and  $f\colon \mC^I\to \mC$  a non-constant homomorphism. Then there exist an ultra filter $\ultra\subseteq \pow{I}$ and an automorphism $\alpha$ of $\mC$ such that for all $a\in C^I$ $$f(a)=\alpha(a_\ultra) \text{ where }a^{-1}(a_\ultra)\in \ultra.$$
\end{lemma}

\begin{proof}
	We assume first that $f(\chi_{\emptyset})=0$ and $f(\chi_{I})=1$. Note that in this case for all $X\subseteq I$ we have $\chi_{\emptyset}\rightarrow \chi_X\rightarrow \chi_I$, therefore $f(\chi_X)\in \{0,1\}$. Let us define $\ultra=\{X\subseteq I: f(\chi_X)=1\}.$ We first show that $\ultra$ is an ultra filter. We prove this by showing that
	
\begin{enumerate}
\item\label{it:central} $\ultra$ is closed under finite intersections,
\item\label{it:max} for all $X\subseteq I$ we have either $X\in \ultra$ or $I\setminus X\in \ultra$,
\item\label{it:empty} $\emptyset\not\in \ultra$.
\end{enumerate}

	Item~\ref{it:empty} is clear from the definition of $\ultra$. For the proofs of the other two items, let $X$ and $Y$ be any two subsets of $I$. We define an equivalence $\varrho$  on $I$ by $$(i,j)\in\varrho\text{ iff }\chi_X(i)=\chi_X(j)\text{ and }\chi_Y(i)=\chi_Y(j).$$ Let $k$ be the number of blocks of $\varrho$. Clearly, $k\leq 4$. Let $\mD$ be the subdigraph spanned by $$D=\{a\in C^I\colon (i,j)\in \rho \text{ implies }a(i)=a(j)\}$$ in $\mC^I$. We require the following claim.
	
\begin{cla} The map $f|_D$ is a projection to some coordinate $s\in I$.\end{cla}

\begin{claimproof} Let $\tau\colon\mC^k\to \mD$	be the isomorphism that maps $(c_1,\dots,c_k)$ to $\sum_{i=1}^k c_i\chi_{B_i}$ where the $B_i$, $1\leq i\leq k$, are the blocks of $\varrho$. Then by Corollary~\ref{cloneofc}, $f|_D\tau\colon \mC^k\to \mC$ is a projection composed with an automorphism of $\mC$. Since $$\chi_\emptyset,\chi_I\in D,\ f(\chi_\emptyset)=0 \text{ and }f(\chi_I)=1,$$ we obtain that $f|_D\tau$ is a projection to some coordinate $j$ where $1\leq j\leq k$. Then $f|_D$ is a projection to any coordinate $s\in B_j$. 
\end{claimproof}
	
To prove item~\ref{it:central}, suppose that $X,Y\in\ultra$. Then by the claim  $f|_D=\pi_s|_D$ for some coordinate $s\in I$. Now $s$ must be in $X\cap Y$, since $\chi_X,\chi_Y\in D$, and $$f(\chi_X)=\pi_{s}(\chi_X)=f(\chi_Y)=\pi_{s}(\chi_Y)=1.$$ Thus $f(\chi_{X\cap Y} )=\pi_{s}(\chi_{X\cap Y} )=1$. 

To prove item~\ref{it:max}, suppose  that $X\not\in \ultra$.  By letting $X=Y$ and applying the claim, we obtain that  there is an $s\in I$ such that $f|_D=\pi_s|_D$. This time $s$ must be in $I\setminus X$, since $f(\chi_X)=\pi_{s}(\chi_X)=0.$  Hence  $f(\chi_{I\setminus X}) =\pi_s(\chi_{I\setminus X})=1$. Thus $\ultra$ is an ultra filter indeed.
	
	From the definition of $\ultra$, for any $a\in \{0,1\}^I$, $f(a)=1$ if and only if $a^{-1}(1)\in \ultra$. If $a\in \{0,1\}^I$, we also have that $f(a)\in \{0,1\}$, so $f(a)=0$ if and only if $a^{-1}(0)\in \ultra$. Let $a\in C^I$  such that its range contains $2$, say $a=\chi_{A}-\chi_{B}$ for two disjoint subsets $A$ and $B$ of $I$. Then the vertices $\chi_{A}+\chi_{B},\ a,\ \chi_{A}$ clearly form a copy of $\mC$ in $\mC^I$. As $\chi_{A}+\chi_{B},\chi_{A}\in \{0,1\}^I$, $f(a)$ is uniquely determined by the known values $f(\chi_{A}+\chi_B)$ and $f(\chi_{A})$. Thus $f$ is the unique homomorphism from $\mC^I$ to $\mC$ extending the partial definition $f(\chi_{\emptyset})=0$ and $f(\chi_{I})=1$. Hence $f$ must be the homomorphism given by $f(a)=a_\ultra$ for any $a\in C^I$.
	
For the general case we distinguish between two cases depending on whether $f(\chi_{\emptyset})=f(\chi_{I})$ or not. We first assume that $f(\chi_{\emptyset})\neq f(\chi_{I})$. As $f$ is a homomorphism, $f(\chi_{I})=f(\chi_{\emptyset})+1 $.  Let $\alpha$ be an automorphism of $\mC$ such that $\alpha\colon f(\chi_{\emptyset})\mapsto 0$. Then $\alpha^{-1}\circ f(\chi_{\emptyset})=0$ and $\alpha^{-1}\circ f(\chi_{I})=1$. Thus $\alpha^{-1}f(a)=a_\ultra$ for some ultra filter $\ultra\subseteq \pow{I}$ which is exactly what we wanted to show.
	
We are left with the case when $f(\chi_{\emptyset})=f(\chi_{I})$. We claim that in this case $f$ must be constant. Let $a$ be an arbitrary element in $C^I$. Put $a=\chi_{A}-\chi_{B}$ for two disjoint subsets $A$ and $B$ of $I$. Let $b=\chi_{A}+ \chi_{B}$ and $c=\chi_A$. Then $$\chi_{\emptyset}\rightarrow b\rightarrow \chi_{I}, \chi_{\emptyset}\rightarrow c\rightarrow \chi_{I}\text{ and }b\rightarrow a \rightarrow c.$$ This is only possible if $f(a)=f(b)=f(c)=f(\chi_{\emptyset})=f(\chi_{I})$.
\end{proof}

The preceding lemma is needed in the present paper  only to verify the following proposition.

\begin{figure}[H] 
\centering
\includegraphics[scale=0.7]{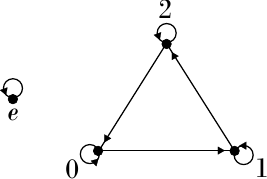}
\caption{The digraph $\mC_1$.} 
\label{fig:cegy}
\end{figure}

\begin{proposition}
 Let $\mC_1$ be the disjoint union of the one-element reflexive digraph and $\mC$ where the isolated element of  $\mC_1$ is denoted by $e$, see Figure \ref{fig:cegy}. Let $\alg C_1$ be an algebra on $C_1=\{e, 0,1,2\}$ whose term operations coincide with the operations of $\pol(\mC_1)$. Then the variety $\var V$ generated by $\alg C_1$ is a non-Taylor variety, and the non-trivial powers of $\mC$ are not compatible digraphs in $\var V$.
\end{proposition}
\begin{proof}
Since the Taylor identities (including idempotency) are preserved under taking retract and $\mC$ that has no Taylor polymorphism is a retract of $\mC_1$, $\var V$ is a non-Taylor variety. To see that no non-trivial power of $\mC$ is a compatible digraph in $\var V$, suppose that $\mC^I$ is compatible in $\var V$ for some $I\neq\emptyset$. Observe that $\mC_1$ admits a binary polymorphism $f$ with a unit element $e$. Then $\mC^I$ must admit a binary polymorphism $f'$ with a unit element. We may conceive of $f'$ as a homomorphism from $\mC^{I\mathbin{\dot\cup}I'}$ to $\mC^I$ where $I'$ is a disjoint copy of $I$. So $f'=(f'_i)_{i\in I}$ where for each $i\in I$, $f'_i$ is a homomorphism from $\mC^{I\mathbin{\dot\cup}I'}$ to $\mC$. Since $f'$ has a unit element, for any $i\in I$, $f'_i$ must depend on each of its coordinates in $I\mathbin{\dot\cup}I'$. Let us fix an $i\in I$. Then by the preceding lemma,  there exist an ultra filter $\ultra\subseteq \pow{I\mathbin{\dot\cup}I'}$ and an automorphism $\alpha$ of $\mC$ such that for all $a\in C^{I\mathbin{\dot\cup}I'}$ $$f'_i(a)=\alpha(a_\ultra) \text{ where }a^{-1}(a_\ultra)\in \ultra.$$ Only one of $I$ and $I'$ is in $\ultra$, say $I\in\ultra$. Let $\vltra$ be the restriction of $\ultra$ to $I$. Clearly, $\vltra $ is an ultra filter on $I$. Moreover,  for any $a\in \mC^{I\mathbin{\dot\cup}I'}$ if $b=a|_I$, then $b_\vltra=a_\ultra$ and $$f'_i(a)=\alpha(b_\vltra) \text{ where }b^{-1}(b_\vltra)\in \vltra.$$ Thus, $f'_i$ only depends on the coordinates in $I$, a contradiction.
\end{proof}

The above proposition just shows that there exist non-Taylor varieties that have no non-trivial compatible $\mC$-powers.
So  the non-Taylor varieties cannot be characterized by the existence of compatible $\mC$-powers. In Section 4, we shall prove that  such a characterization is possible by the use of a proper class of compatible disjoint unions of $\mC$-powers. In the next section, we argue that there is no such characterization given by a single disjoint union of finite $\mC$-powers. The proof of this fact uses a rather interesting Maltsev characterization of the $n$-th powers of relational structures. 

\section{A Maltsev condition characterizing the $n$-th powers of relational structures}

In this section, relational structures are allowed to have infinitary relations. For any relational structure $\mA$ and equivalence $\nu$ of $A$, we define the quotient relational structure $\mA/\nu$ of the same signature as $\mA$ as follows.
The base set of $\mA/\nu$ equals the set $A/\nu$ of the blocks of  $\nu$, and for any $I$-ary relation $R$ of $\mA$ the corresponding $I$-ary relation of $\mA/\nu$ consists of the $I$-tuples $(a_i/\nu)_{i\in I}$ where $(a_i)_{i\in I}\in R$.

An operation $f\colon A^n\rightarrow A$ is called an {\em $n$-ary product decomposition operation} if $f$ satisfies the following identities 
\begin{align*}
f(x,\dots,x)=& x,\\
f(f(x_{1,1},\dots,x_{1,n}),\dots, f(x_{n,1},\dots,x_{n,n}))=& f(x_{1,1},\dots,x_{n,n}).
\end{align*}

\begin{remark}
\emph{Let $\mA= \prod_{j=1}^n\mB_j$. Then the operation $f:A^n\rightarrow A$ defined by
$$f((b_{1,1},\dots,b_{1,n}),\dots,(b_{n,1},\dots,b_{n,n}))=(b_{1,1},\dots,b_{n,n}) \text{ where } i,j\leq n,\ b_{i,j}\in B_j$$
is clearly an $n$-ary product decomposition polymorphism of $\mA$.}
\end{remark}

Perhaps after this remark, it is not surprising that the existence of an $n$-ary product decomposition polymorphism leads to an $n$-fold product decomposition of a relational structure as follows.

\begin{theorem}\label{prodecomp}
	Let $\mA$ be any relational structure which has an $n$-ary product decomposition polymorphism $f$. For each $i\leq n$, let $\nu_i$ denote the binary relation that consists of the pairs $(a,b)\in A^2$ such that for some sequence $c_j$, $j\neq i$, in $A$
	$$f(c_{1},\dots, c_{i-1}, a, c_{i+1},\dots,c_{n})=f(c_{1},\dots, c_{i-1}, b, c_{i+1},\dots,c_{n}).$$
	Then the following hold.
	\begin{enumerate}
	\item For each $i\leq n$, $\nu_i$ is an equivalence on $A$.
	\item The map $$\varphi\colon A\rightarrow \prod_{i=1}^n{{A}/{\nu_i}},\  a\mapsto (a/{\nu_1},\dots,a/{\nu_n}) $$ is an isomorphism from $\mA$ to the relational structure $\prod_{i=1}^n{{\mA}/{\nu_i}}$. 
	\end{enumerate}
	\end{theorem}
\begin{proof} First we prove item (1). Let $i\leq n$. Clearly, $\nu_i$ is reflexive and symmetric. For proving transitivity, we require the following claim.

\begin{cla} For each $i\leq n$, $\nu_i$ consists of the pairs $(a,b)\in A^2$ such that for any sequence $d_j$, $j\neq i$, in $A$
	$$f(d_{1},\dots, d_{i-1}, a, d_{i+1},\dots,d_{n})=f(d_{1},\dots, d_{i-1}, b, d_{i+1},\dots,d_{n}).$$\end{cla}

\begin{claimproof} Suppose that $a\nu_i b$. Then there exist $c_j$, $j\neq i$, such that $$f(c_{1},\dots, c_{i-1}, a, c_{i+1},\dots,c_{n})=f(c_{1},\dots, c_{i-1}, b, c_{i+1},\dots,c_{n}).$$ By using the fact that $f$ is an $n$-ary product decomposition operation, we obtain that for any  $d_j$, $j\leq n$,
\begin{align*}
f(f(d_1,d_2,\dots,d_n),\dots, f(c_{1},\dots, c_{i-1}, a, c_{i+1},\dots,c_{n}), \dots, f(d_1,d_2,\dots,d_n))=\\ 
f(d_{1},\dots, d_{i-1}, a, d_{i+1},\dots,d_{n})
\end{align*} and
\begin{align*}
f(f(d_1,d_2,\dots,d_n),\dots, f(c_{1},\dots, c_{i-1}, b, c_{i+1},\dots,c_{n}), \dots, f(d_1,d_2,\dots,d_n))=\\ 
f(d_{1},\dots, d_{i-1}, b, d_{i+1},\dots,d_{n}).
\end{align*} 

Since the left sides of the preceding two equalities are equal, so are the right sides. This concludes the proof of the claim. \end{claimproof}

Suppose now that  $a\nu_i b\nu_i c$. Then by the claim, for any $d_j,\ j\neq i$, 
\begin{align*} 
f(d_{1},\dots, d_{i-1}, a, d_{i+1},\dots,d_{n})=&f(d_{1},\dots, d_{i-1}, b, d_{i+1},\dots,d_{n})=\\&f(d_{1},\dots, d_{i-1}, c, d_{i+1},\dots,d_{n}).
\end{align*}
In particular, $a\nu_i c$. Thus $\nu_i$ is transitive.

We prove item (2). Suppose that $\varphi(a)=\varphi(b)$. Then  for all $i\leq n$, $a\nu_i b$. Since $f$ is idempotent, by applying the above claim for the pair $(a,b)\in \nu_i$, $d_1=\dots =d_i=b$ and $d_{i+1}=\dots =d_n=a$ where $i\leq n$ we obtain   
$$a=f(a,a,\dots,a)=f(b,a,\dots,a)=f(b,b,a,\dots,a)=\dots=f(b,\dots,b)=b.$$ So $\varphi$ is injective. 


To prove that $\varphi$ is surjective, it suffices to show that for any sequence $a_i,\ i\leq n,$ in $A$,  
$\varphi(f(a_1,\dots,a_n))=({a_1}/{\nu_1},\dots,{a_n}/{\nu_n})$, that is, for all $i$, $f(a_1,\dots,a_n)\nu_i a_i$. We prove the latter for $i=1$, the proof is analogous for the cases when $i\geq 2$. By the properties of  $f$ 
\begin{align*}
f(f(a_1,\dots,a_n),a_2,\dots, a_n)=f(f(a_1,\dots,a_n),f(a_2,\dots,a_2),\dots, f(a_n,\dots,a_n))=\\ 
f(a_1,\dots,a_n). 
\end{align*}
So $f(a_1,\dots,a_n)\nu_1 a_1$ indeed.

Clearly, $\varphi$ preserves the relations of $\mA$. We prove that its inverse also does. Let $\varrho$ be a $J$-ary relation of $\mA$, and suppose that $(\varphi(a_j))_{j\in J}\in \varrho'$ where $\varrho'$ is the relation of $\prod_{i=1}^n{{\mA}/{\nu_i}}$ corresponding to $\varrho$. By the definition of $\varrho'$, for each $i\leq n$ there is a tuple $(a_{j,i})_{j\in J}\in \varrho$ such that $a_j \nu_i a_{j,i}$ for all $j\in J$. Since $f$ preserves $\varrho$, $(f(a_{j,1},\dots,a_{j,n}))_{j\in J}\in \varrho$. In the preceding paragraph we saw that $\varphi(f(a_{j,1},\dots,a_{j,n}))=({a_{j,1}}/{\nu_{j,1}},\dots,{a_{j,n}}/{\nu_{j,n}})$. So by
$a_j \nu_i a_{j,i}$, $i\leq n$, we have that $\varphi(a_j)=\varphi(f(a_{j,1},\dots,a_{j,n}))$. Since $\varphi$ is injective, $a_j=f(a_{j,1},\dots,a_{j,n})$ for all $j\in J$. Thus $(a_j)_{j\in J}\in \varrho$, which concludes the proof.
\end{proof}

Our main goal in this section is to obtain a characterization of the $n$-th powers of relational structures via some Maltsev condition.
In order to do this, we require the following definition. An operation $f\colon A^n\rightarrow A$ is called an {\em $n$-ary power decomposition operation} if $f$ is an $n$-ary product decomposition 
operation and there is a $g\colon A\rightarrow A$ such that $f$ and $g$ satisfy the additional identities
\begin{align*}
g^n(x)=& x,\\
g(f(x_{1},x_{2},\dots,x_{n}))=& f(g(x_2),\dots,g(x_n),g(x_1)).
\end{align*}
Such a $g$ is called a {\em coordinate shift operation} with respect to $f$. In this case we also say that $f$ is an $n$-ary power decomposition operation with respect to $g$.

\begin{remark}
\emph{Let $\mA= \mB^n$. Then the operation $f:A^n\rightarrow A$ defined by
$$f((b_{1,1},\dots,b_{1,n}),\dots,(b_{n,1},\dots,b_{n,n}))=(b_{1,1},\dots,b_{n,n}) \text{ where } i,j\leq n,\ b_{i,j}\in B$$
is clearly an $n$-ary power decomposition polymorphism of $\mA$ with the coordinate shift automorphism $g$ defined by
$$g(b_{1},b_2,\dots,b_n)=(b_{2},\dots,b_n,b_1) \text{ where } i\leq n,\ b_{i}\in B.$$}
\end{remark} 

Now we prove that the existence of an $n$-ary power decomposition polymorphism with repect to a coordinate shift endomorphism implies that the relational structure is isomorphic to the $n$-th power of some relational structure.

\begin{theorem}\label{powdecomp}
	Let $\mA$ be any relational structure which has an $n$-ary power decomposition polymorphism $f$ with respect to a coordinate shift endomorphism $g$. For each $i\leq n$, let $\nu_i$ denote the equivalence defined in the statement of the preceding theorem. 
Then $g$ is an automorphism of $\mA$, and for any $i\leq n$, $g(\nu_i)=\nu_{i-1}$ where $i-1$ is meant by modulo $n$. Moreover,  $\mA$ is isomorphic to the $n$-th power of the relational structure ${\mA}/{\nu_1}$.
\end{theorem}
\begin{proof} Since $g$ is an endomorphism with $g^n(x)=x$, $g$ is an automorphism of $\mA$ with inverse $g^{n-1}$.  
Let $i\leq n$. Let $(a,b)\in \nu_i$. we prove that $(g(a),g(b))\in \nu_{i-1}$. By $(a,b)\in \nu_i$, there are some elements
$c_i\in A, i\leq n$, such that $$f(c_1,\dots, c_{i-1},a, c_{i+1},\dots, c_n)=f(c_1,\dots, c_{i-1},b, c_{i+1},\dots, c_n).$$ By plugging both sides into $g$ and applying the identity satisfied by $f$ and $g$ we obtain 
\begin{align*}
&f(g(c_2),\dots, g(c_{i-1}),g(a), g(c_{i+1}),\dots, g(c_n), g(c_1))=\\ 
&f(g(c_2),\dots, g(c_{i-1}),g(b), g(c_{i+1}),\dots, g(c_n), g(c_1)).
\end{align*}
Thus $(g(a),g(b))\in \nu_{i-1}$, and hence $g(\nu_i)\subseteq \nu_{i-1}$ for all $i$. Then
$$\nu_{i-1}=g^n(\nu_{i-1})\subseteq g^{n-1}(\nu_{i-2})\subseteq\dots\subseteq g^2(\nu_{i+1}) \subseteq g(\nu_i)\subseteq \nu_{i-1}$$
where the indices are meant by modulo $n$. Therefore, $g(\nu_i)=\nu_{i-1}$ for all $i\leq n$. 

Finally, we prove that all of the factors $\mA/{\nu_i}$, $i\leq n$, of the product decomposition of $\mA$ are isomorphic to each other. Indeed, since $g$ is an automorphism, for all $i\leq n$                       
\[ \mA/{\nu_i}\cong g(\mA)/{g(\nu_i)}\cong \mA/{\nu_{i-1}}. \qedhere\]
\end{proof}

As an immediate consequence of the preceding remark and theorem, we obtain a characterization of the $n$-th powers by a strong Maltsev condition.

\begin{theorem}\label{powerdecomp} 
Any relational structure $\mA$ is the $n$-th power of a relational structure if and only if $\mA$ admits the strong Maltsev condition
\begin{align*}
f(x,\dots,x)=& x,\\
f(f(x_{1,1},\dots,x_{1,n}),\dots, f(x_{n,1},\dots,x_{n,n}))=& f(x_{1,1},\dots,x_{n,n}),\\
g^n(x)=& x,\\
g(f(x_{1},x_{2},\dots,x_{n}))=& f(g(x_2),\dots,g(x_n),g(x_1)).
\end{align*}
\end{theorem}

We remark that for any $n$, the Maltsev condition that characterizes the $n$-fold product decomposition (the set of the first two identities for $f$ in the corollary) is a trivial Maltsev condition. Indeed, by taking $f$ as the first projection on the two element set, we obtain an $n$-ary product decomposition operation.

On the other hand, if a relational structure $\mA$ of at least two elements admits the Maltsev condition for the $n$-th power decomposition (the set of all four identities in the corollary), then $f$ must depend on each of its coordinates. This is because the equivalences $\nu_i$, $i\leq n$, corresponding to $f$ differ from the full relation $A^2$. In particular, for $n\geq 2$, the Maltsev condition that describes the $n$-th power decomposition is a non-trivial Maltsev condition. (This also follows from the fact that a 2-element set never is an $n$-th power for $n\geq 2$.)

Now we apply Theorem \ref{powerdecomp} to verify that certain digraphs are forbidden as compatible digraphs in the variety assigned naturally to the $n$-th power of a finite digraph of certain type.

\begin{corollary} 
Let $n\geq 2$, and let $\mA$ be a finite, reflexive, connected, directly indecomposable  digraph. Let $\alg A_n$ be the algebra on $A^n$ with the polymorphism of $\mA^n$ as basic operations, and $\var V$ the variety generated by $\alg A_n$.
Then $\var V$ has no compatible digraph of the form $\dot\bigcup_{j\in J} \mA^{n_j}$ where $n_j<n$ for all $j\in J$ and  $n_j>0$ for some $j\in J$.
\end{corollary}
\begin{proof}
Let us suppose to the contrary that $\var V$ has a compatible digraph of the form $\dot\bigcup_{j\in J} \mA^{n_j}$. Since $\mA^n$ and so  $\var V$ admit the Maltsev condition given in the preceding theorem, by the theorem there is a digraph $\mB$ such that $\mB^n\cong\dot\bigcup_{j\in J} \mA^{n_j}$. Now by the assumption for the disjoint union, at least one of the connected components of $\mB$, say $\mB_0$ has more than one element. Since the disjoint union is reflexive, $\mB$ and all of its components are reflexive. Hence $\mB_0^n$ is a component of $\mB^n$. So $\mB_0^n\cong \mA^{n_j}$ for some ${n_j}>0$. Hence by the unique decomposition theorem for finite connected reflexive digraphs, see Corollary 4.7 in \cite{M}, $\mB_0\cong \mA^{m}$ for some $m\geq 1$. This implies that $|A|^{mn}=|B_0|^n= |A|^{n_j}$ which contradicts the fact that $n>n_j$.
\end{proof}

The immediate corollary below shows indeed that there exists no disjoint union $\mU$ of finite $\mC$-powers such that some of the $\mC$-powers are non-trivial and  $\mU$ is a compatible digraph in any non-Taylor variety. In other words, the non-Taylor varieties cannot be characterized by the existence  of a fixed compatible disjoint union of finite $\mC$-powers.

\begin{corollary} 
Let $n\geq 2$. Let $\alg C_n$ be the algebra on $C^n$ with the polymorphism of $\mC^n$ as basic operations, and $\var V$ the variety generated by $\alg C_n$.
Then $\var V$ has no compatible digraph of the form $\dot\bigcup_{j\in J} \mC^{n_j}$ where $n_j<n$ for all $j\in J$ and  $n_j>0$ for some $j\in J$.
\end{corollary}


\section{Characterizations of Taylor varieties}
In this section, we prove the main result of the paper. In order to do this, we require the following Lemma. 

\begin{lemma}\label{clonehom} Let $\mG$ be a relational structure and $\mK\subseteq\mG$. Let $S$ be subclone of $\pol(\mG)$. If for any $f\in \pol(\mK)$, there is unique extension $f^*\in S$, then the map $$\varphi\colon \pol(\mK)\to  S,\ f\mapsto f^*$$ is a clone homomorphism.
\end{lemma}  
\begin{proof}
Clearly, any projection operation $f$ on $\mK$ extends to a projection operation of $\mG$. By uniqueness of $f^*$, the latter projection operation must be $f^*$. We prove that $\varphi$ commutes with composition. Let $f$ be an $n$-ary and $g$ be an $m$-ary operation in 
$\pol(\mK).$ Then the restriction of $f^*(g^*(x_1,\dots,x_m),y_2,\dots, y_n)$ onto $\mH$ is $f(g(x_1,\dots,x_m),y_2,\dots, y_n)$. So by uniqueness of the extension in $S$
\[(f(g(x_1,\dots,x_m),y_2,\dots, y_n))^*=f^*(g^*(x_1,\dots,x_m),y_2,\dots, y_n). \qedhere\]\end{proof}

The following characterization of non-Taylor varieties plays a crucial role in the proof of our main result.

\begin{lemma}\label{majdnem_jo}
Let $\var V$ be a variety. Then $\var V$ is a non-Taylor variety if and only if there exist a non-empty  set $T$ and sets $H_t,\ t\in T$, such that some of the $H_t$ are non-empty and  $\dot\bigcup_{t\in T}\mC^{H_t}$ is a compatible digraph in $\var V$. If  $\var V$ is a locally finite non-Taylor variety, then $T$ and $H_t,\ t\in T,$ can be chosen to be finite.
\end{lemma}

\begin{proof}
	

 
First	 we prove the ``only if'' part of the first statement of the  lemma. So we assume that $\var V$ is a non-Taylor variety. Let  $\alg F$ be the free algebra freely generated by three elements $x,y$ and $z$ in $\var V$. We define a compatible digraph $\mF$ of $\alg F$ with letting the edge relation of $\mF$ be the subalgebra of $\alg F^2$ generated by the set of pairs $$\{(x,x),\ (y,y), \ (z,z),\ (x,y), \ (y,z), \ (z,x)\}.$$  By definition, for any $u,v\in F$, $u\to v$ in $\mF$ if and only if there is a 6-ary term $h$ of $\var V$ such that $u(x,y,z)=h(x,y,z,x,y,z) $ and $v(x,y,z)=h(x,y,z,y,z,x)$ in $\alg F$.

	Clearly, $x,y$ and $z$ are pairwise different elements in $F$. Indeed, being non-Taylor, $\var V$ does not satisfy the identity $x=y$. Notice that the vertices $x,y$ and $z$ with the six edges of the above generating set constitute a subdigraph of $\mF$ isomorphic to $\mC$. This subdigraph of $\mF$ is called $\mC_\mF$. The subdigraph $\mC_\mF$ is a spanned subdigraph of $\mF$, for otherwise, there would be a 6-ary term $h$ such that $$(h(x,y,z,x,y,z),h(x,y,z,y,z,x))\in \{(y,x),(z,y),(x,z)\}$$ which is impossible, because the resulting identities, for example $$h(x,y,z,x,y,z)=y\text{ and } h(x,y,z,y,z,x)=x,$$  make $h$ to be a Taylor term for $\var V$.

We conceive of the underlying set $F$ of $\alg F$ as the set of ternary term operations of $\alg F$.  Let $T$ be the set of unary term operations of $\alg F$. For any $t\in T$, we write $F_t$ for the set of the ternary term operations $u\in F$ such that $u(x,x,x)=t(x)$. Let $\mF_t$ be the subdigraph of $\mF$ spanned by $F_t$, $t\in T$. We claim that $\mF_t,\ t\in T,$ are exactly the connected components of $\mF$. Indeed, if $u\to v$ in $\mF$, then $$u(x,x,x)=h(x,x,x,x,x,x)=v(x,x,x)=t(x)$$ for some 6-ary term $h$, and thus each connected component is contained in some $F_t$. On the other hand, by using the compatibility of $\mF$ in $\alg F$, for any $t\in T$ and $u\in F_t$, we have $u(x,y,z)\to u(x,y,x)$ and $u(x,x,x)\to u(x,y,x)$. Hence for all $u\in F_t$,  $u$ and $t(x)=u(x,x,x)$ are in the same component of $\mF$. Thus, for any $t\in T$, $\mF_t$ is connected.

For any $t\in T$, let $H_t$ denote the set of non-constant homomorphisms from $\mF_t$ to $\mC$. Let $T_0\coloneqq \{t\in T\colon H_t=\emptyset \}$ and $T^+\coloneqq T\setminus T_0$. For any $t\in T_0$, we define $0_t$ to be  a new constant  and for any $u\in F_t$, we let $[u]=0_t$. For any $t\in T^+$ and for any $u\in F_t$, we let $[u]$ be the function $H_t\rightarrow \mC, \varphi\mapsto \varphi(u)$.  

For any $t\in T$, we define the digraph $\mK_t$ as follows. The vertices of $\mK_t$ are the $[u],\ u\in F_t,$ and the edges of $\mK_t$ are the $([u],[v])$ where $u\to v$ in $\mF_t$. We remark that by the definition, for any $t\in T$, $t\in T_0$ if and only if $\mK_t$ is a one-element digraph, in which case the vertex set of $\mK_t$ is $\{0_t\}$ and its edge relation is $\{(0_t,0_t)\}$. We define $\mK$ to be the  digraph whose connected components coincide with the digraphs $\mK_t$, $t\in T$, that is, $\mK$ is the disjoint union of the $\mK_t$, $t\in T$.   Clearly, $\mK$ is a reflexive subdigraph of the digraph  $\mG\coloneqq\dot\bigcup_{t\in T}\mC^{H_t}$ where the copy of $\mC^{H_t}$ is identified with $\mK_t$ whenever $t\in T_0$. We prove the required properties of the digraph $\mG$ by the help of the next four claims. 
	
\begin{cla1}The kernel of the map $\psi\colon F\mapsto K, u\mapsto [u]$ is a congruence of $\alg F$.\end{cla1}  

\begin{claimproof} To verify Claim 1, it suffices to show that $\ker(\psi)$ is preserved by every unary polynomial operation of $\alg F$. Every unary polynomial operation of $\alg F$ is of the form $s(x,u_2,\dots,u_n)$ where $u_2,\dots,u_n\in F$ and $s$ is an $n$-ary term operation of $\alg F$. So let us assume that $(u,v)\in \ker(\psi)$. Then $u,v\in F_t$ for some $t$. We prove that $(s(u,u_2,\dots,u_n),s(v,u_2,\dots,u_n))$ is also in $\ker(\psi)$. As $\mF$ is reflexive, the digraph $\hat\mF_t=\mF_t\times\{u_2\}\times\dots\times \{u_n\}$ is isomorphic to $\mF_t$, and hence $\hat\mF_t$ is connected. Since $\hat\mF_t\subseteq\mF^n$ and $s$ is a polymorphism of $\mF$, $s(\hat\mF_t)$ is also connected. Clearly,
$s(u,u_2,\dots,u_n),s(v,u_2,\dots,u_n)\in s(\hat F_t)$. Hence, there exists $t'\in T$ such that $s(u,u_2,\dots,u_n),s(v,u_2,\dots,u_n)\in F_{t'}$. In the case when $t'\in T_0$, this yields immediately that $(s(u,u_2,\dots,u_n),s(v,u_2,\dots,u_n))\in  \ker(\psi)$. Now suppose that $t'\in T^+$, and let $\varphi\in H_{t'}$ be arbitrary. Then the map $$\nu\colon F_t\mapsto \mC, x\mapsto \varphi(s(x,u_2,\dots,u_n))$$ is a homomorphism from $\mF$ to $\mC$. Since $(u,v)\in \ker(\psi)$, this implies in particular that $\nu(u)=\nu(v)$. Therefore $\varphi(s(u,u_2,\dots,u_n))=\varphi(s(v,u_2,\dots,u_n))$. Since this holds for all $\varphi\in H_{t'}$, we obtain that  $(s(u,u_2,\dots,u_n),s(v,u_2,\dots,u_n))\in  \ker(\psi)$. \end{claimproof}

	Now, we define an algebra $\alg K$ in $\var V$ whose underlying set is $K$ and whose basic operations are defined by $$u_{\alg K}([x_1],\dots,[x_n])\coloneqq [u_{\alg F}(x_1,\dots,x_n)]$$ for all basic operations $u_{\alg F}$ of $\alg F$ and $x_1,\dots,x_n\in F$ provided $u_{\alg F}$ is $n$-ary. By Claim~1, the basic operations of $\alg K$  are well-defined. Now it is clear that $\psi$ is an onto homomorphism from $\alg F$ to $\alg K$ and $\alg K\cong \alg F/{\ker(\psi)}$.   Moreover, it is also clear from our construction that $\mK$ is compatible with $\alg K$, and $\mK\cong\mF/{\ker(\psi)}$.

\begin{cla2} The vertices $[x],[y],[z]$ are pairwise different.\end{cla2} 

\begin{claimproof} As $\var V$ is non-Taylor, its full idempotent reduct $\var V_{\Id}$ interprets in the variety  $\var S\var E\var T$. Let $\id\in T$ denote the identity term operation of $\alg F$. Recall that $F_{\id} $ is the set of all ternary idempotent term operations of $\alg F$. We define two algebras $\alg A$ and $\alg B$ on the same underlying set $F_{\id}$.  Algebra $\alg A$ is defined to be an algebra in $\var V_{\Id}$ whose term operations are the restrictions of the idempotent term operations of $\alg F$ to $F_{\id}$.  Algebra $\alg B$ is defined to be the unique algebra on $F_{\id}$ in $S\var E\var T$. Notice that $\alg A$ is the free algebra freely generated by three elments  in $\var V_{\Id}$.  Then, a set of identities  in two variables that holds for the free algebra $\alg A$, also holds for $\var V_{\Id}$. In particular, $\alg A$ has no Taylor term operation, since $\var V_{\Id}$ interprets in $\var S\var E\var T$ .  So there is a clone homomorphism $\zeta$ from $\clo(\alg A)$ to $\clo(\alg B)$. We know that the vertices of $\mF_{\id}$ are represented by the the idempotent ternary term operations of $\alg F$, and the edges of $\mF_{\id}$ are given by certain identities of $\var V_{\Id}$ of the form $$u(x,y,z)=h(x,y,z,x,y,z) \text{ and }v(x,y,z)=h(x,y,z,y,z,x).$$ Now $\zeta$ preserves these identities, maps the ternary term operations of $\alg A$ to ternary projections and fixes the ternary projections on $F_{\id}$. Hence $\zeta$ restricted to $F_{\id}\subseteq \clo(\alg A)$ is a retraction from $\mF_{\id}$ onto $\mC_\mF$. This retraction composed with an isomorphism from $\mC_\mF$ to $\mC$  yields a homomorphism from $\mF_{\id}$ to $\mC$. Now this homomorphism is in $H_{\id}$ and separates the elements $x,y,z$. Hence  $[x],[y],[z]\in K$ are pairwise different as we claimed. \end{claimproof}
	
	Note that in $\mK$ we have the edges $[x]\to [y],\ [y]\to [z],\ [z]\to [x]$, as well as all the loops. Also, since $\mK_{\id}\subseteq\mC^{H_{\id}}$, there are no more edges between the vertices $[x],[y],[z]$ in $\mK$. This means in other words that the subdigraph induced on $\{[x],[y],[z]\}$ is isomorphic to $\mC$. We write $\mC_{\mK}$ for this subdigraph. 

\begin{cla3} For any $t\in T$,  $t\in T^+$ if and only if $|\{[t(x)],[t(y)],[t(z)]\}|=3$.\end{cla3}  

\begin{claimproof} The ``if'' direction of Claim 3 is obvious. Let us assume that $t\in T^+$, i.e. there exists a non-constant homomorphism $\varphi$ from $\mF_t$ to $\mC$. Then $\varphi(u(x,y,z))\neq \varphi(t(x))$ for some $u\in F_t$. So $[u(x,y,z)]\neq [t(x)]= [u(x,x,x)]$. Let $u'$ denote the restriction of $u$ to $\mC_{\mK}^3$. Since $\mK$ is compatible in $\alg K$, $u'$ is a homomorphism from $\mC_{\mK}^3$ to $\mK_t\subseteq \mC^{H_t}$, and $u'$ is not constant since
$[u(x,y,z)]= u([x],[y],[z])$ and $[u(x,x,x)]= u([x],[x],[x])$ are contained in the image $u(\mC_{\mK}^3)$ of $u'$. 
Then by Lemma~\ref{image_of_powers},  $u'=\iota\pi_J$ where $\pi_J$ is the projection to the coordinates in $J$ for some non-empty $J$,  and $\iota$ is an isomorphism.   Then $\pi_J$ and hence $u'$ map the diagonal elements of $\mC_{\mK}^3$ to separate elements. So $[t(x)]=u([x],[x],[x]),\ [t(y)]=u([y],[y],[y])$, and $[t(z)]=u([z],[z],[z])$ are pairwise different. \end{claimproof}

 Now let $f$ be an $n$-ary operation in $\pol(\mK)$. Clearly, $f$ restricted to any connected component of $\mK^n$ is a homomorphism to some component of $\mK$. It is also clear that any component $\mD$ of $\mK^n$ equals some product $\prod\limits_{j=1}^n \mK_{t_{j}}$ where $t_1,\dots,t_n\in T$. So $f|_D$ maps $\prod\limits_{j=1}^n \mK_{t_{j}}$ to some $\mK_t,\ t\in T$. 
Clearly, if $t\in T_0$, since  $\mK_t$ is one-element, $f|_D$ is a constant map. If $t\in T^+$, since $\mK_t\subseteq \mC^{H_t}$, $f|_D=(f_s)_{s\in H_t}$ where the $f_s$ are homomorphisms from $\prod\limits_{j=1}^n \mK_{t_j}$ to $\mC$. 
	
\begin{cla4} If $t\in T^+$, then any of the homomorphisms $f_s$ as above is either constant or there exist a unique $1\leq \ell\leq n$ with $t_{\ell}\in T^+$ and a unique homomorphism $\varphi\in H_{t_{\ell}}$ such that $f_s(y_1,\dots,y_n)=y_{\ell}(\varphi)$.\end{cla4}
	
\begin{claimproof} For the proof of Claim 4, let $$\mP\coloneqq t_1(\mC_{\mK})\times \dots \times t_n(\mC_{\mK}).$$ Then we have $\mP\subseteq \prod\limits_{j=1}^n \mK_{t_{j}}$. By Claim 3, we know that if $t_{j}\in T^+$ then $t_{j}(\mC_{\mK})$ is isomorphic to $\mC$, otherwise $t_{j}(\mC_{\mK})$ is a singleton. Therefore, $\mP$ is isomorphic to a power of $\mC$. Hence, by Corollary~\ref{cloneofc}, the restriction of $f_s$ to $P$ either is a constant map or there is a unique $\ell$ such that $t_{\ell}\in T^+$ and $f_s|_P$ only depends on the $\ell$-th coordinate.

Let $([u_1],\dots,[u_n])$ be any tuple in $\prod\limits_{j=1}^n \mK_{t_{j}}$ where the $u_j$ are arbitrary elements in $F_{t_j}$ for $1\leq j\leq n$. Note that in this case $$\mP\subseteq u_1(\mC_\mK^3)\times\dots\times u_n(\mC_\mK^3)\subseteq \prod\limits_{j=1}^n \mK_{t_{j}}.$$ By Lemma~\ref{image_of_powers}, we know that each $u_1(\mC_\mK^3)$ is isomorphic to a finite power of $\mC$, therefore $$u_1(\mC_\mK^3)\times\dots\times u_n(\mC_\mK^3)$$ is also isomorphic to a finite power of $\mC$. Hence, by Corollary~\ref{cloneofc}, the restriction of $f_s$ to $$u_1(C_\mK^3)\times\dots\times u_n(C_\mK^3)$$ is a constant map or there is an $1\leq\ell'\leq n$ such that $t_{\ell'}\in T^+$ and the restriction only depends on the $\ell'$-th coordinate. Since $P\subseteq u_1(C_\mK^3)\times\dots\times u_n(C_\mK^3)$, it follows that $\ell'=\ell$,  that is, $\ell'$ does not depend on our choice of the tuple $([u_1],\dots,[u_n])$. 
	
	This implies, that either $f_s$ is constant, or there exists a unique $\ell$ such that $t_{\ell}\in T^+$ and $f_s(y_1,\dots,y_n)=q(y_l)$ where $q$ is a non-constant map from $\mK_{t_{\ell}}$ to $\mC$. In the latter case,  $q$ must be a homomorphism, since $f_s$ is a homomorphism from the product of reflexive digraphs $\mK_{t_j}$, $1\leq j\leq n$.  Let $\varphi\coloneqq q\circ \psi$. Since $\psi(F_{t_\ell})=K_{t_\ell}$, $\varphi$ is not constant, and hence
$\varphi\in H_{t_\ell}$. Then $q$ must be the projection to the coordinate $\varphi$. Indeed,
for any element $u\in F_{t_\ell}$ $$q([u])=q(\psi(u))=\varphi(u)=[u](\varphi).$$ The uniqueness of $\varphi$ is now obvious.\end{claimproof}

By Claim 2, $H_{\id} \neq \emptyset$. So to finish the proof of the ``only if'' part of the first statement of the  lemma, we prove that $\mG= \dot\bigcup_{t\in T}\mC^{H_t}$ is a compatible digraph in $\var V.$ As $\mK$ is compatible in $\var V$, it suffices to verify that there is a clone homomorphism from $\pol(\mK)$ to $\pol(\mG).$ We prove the existence of such a clone homomorphism by an application of Lemma \ref{clonehom} as follows. Clearly, $\mK\subseteq \mG$. We define a subclone $S$ of $\pol(\mG).$ The clone $S$ consists of the $n$-ary operations $g\in \pol(\mG))$ such that for each component $\mD$ of the digraph $\mG^n$ the restriction $g|_D$ is either constant or $g(D)\subseteq C^{H_t}$ for some $t\in T^+$ and each of the coordinate maps of $g|_D$ is a projection to some coordinate or constant. It is easy to see that $S$ is indeed a subclone. By Claim 4, every $n$-ary operation of $f\in\pol (\mK)$ uniquely extends to an operation of $f^*\in S$. So by Lemma \ref{clonehom}, it follows that there exists a clone homomorphism from $\pol (\mK)$ to $S$, and hence to $\pol(\mG).$

To conclude the proof of the first statement of the lemma, we assume that $\dot\bigcup_{t\in T}\mC^{H_t}$ is a compatible digraph in $\var V$ and for some $t\in T$, $H_t\neq\emptyset$. Then $\mC$ is a retract of $\dot\bigcup_{t\in T}\mC^{H_t}$. As the existence of a Taylor polymorphism is inherited under taking retract and  $\mC$ admits no Taylor operation, $\var V$ is a non-Taylor variety.

The finiteness statement of the lemma  is obvious by our construction of the $\mC$-power components of $\mG$ in the above proof.
\end{proof}

In order to prove the primeness of the filter of Taylor interpretability types in $\LL$, we require a modified version of the preceding lemma.

\begin{corollary} \label{kappa} Let $\var V$ be a variety. Then $\var V$ is a non-Taylor variety if and only if
for any large enough cardinal $\kappa$, the digraph $\dot\bigcup_{\lambda\leq \kappa}2^\kappa\mC^{\lambda}$, where $\lambda$ stands for a cardinal and $2^\kappa\mC^{\lambda}$ denotes the disjoint union of $2^\kappa$-many copies of $\mC^{\lambda}$, is a compatible digraph in $\var V$. 
\end{corollary}

\begin{proof}
We prove the ``only if'' part of the claim, the other direction is immediate from the preceding lemma.
Let $\var V$ be an arbitrary non-Taylor variety. Since $\var V$ is a non-Taylor variety, by the preceding lemma, there are some cardinals $\kappa_t,\ t\in T,$ not all zero, such that $\mG=\dot\bigcup_{t\in T}\mC^{\kappa_t}$ is a compatible digraph in $\var V$. Let $\kappa$ be  an arbitrary infinite cardinal with $\kappa\geq\kappa_t$ for all $t\in T$ and $\kappa\geq |T|$. By taking  the $\kappa$-th power of  $\mG$, we obtain a compatible digraph $\mG_1$ in $\var V$ such that $\mG_1$ is also a disjoint union of $\mC$-powers and the largest exponent that appears in the $\mC$-power components of $\mG_1$ is $\kappa.$ Clearly, $\mG_1$ has at most $2^\kappa$ components and $|\mG_1|=2^\kappa$.

By a primitive positive definition as follows, we construct a compatible digraph $\mG_2$ from $\mG_1$ in $\var V$ such that  $\mG_2$ is a sum of $\mC$-powers where each cardinal $\lambda\leq\kappa$ appears as an exponent of some $\mC$-power component. The vertex set of $\mG_2$ is the set of homomorphisms from $\mC$ to $\mG_1$. To define the edge set of $\mG_2$ we use the digraph $\mD$ given in Figure \ref{fig:gadget}. We set $f\to g$ in $\mG_2$ iff the map given  by $a_i\mapsto f(i)$, $b_i\mapsto g(i)$, $0\leq i\leq 2$, is a homomorphism from $\mD$ to $\mG_1.$

\begin{figure}[H] 
\centering
\includegraphics[scale=1]{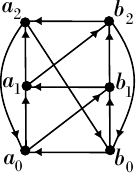}
\caption{The digraph $\mD$.} 
\label{fig:gadget}
\end{figure}

It is easy to see that $\mG_2$ is a reflexive digraph. Since $\mD$ is connected, any connected component of $\mG_2$ is contained in the set of the homomorphisms from $\mC$ to $\mC^\nu$ where  $\mC^\nu$ is a connected component of the digraph $\mG_1$. Let $(d_0,d_1,d_2)\in G_2$ an arbitrary element, that is, either $\{d_0,d_1,d_2\}$ is a singleton or induces a subdigraph isomorphic to $\mC$ in $\mG_1$.
Our goal is to describe the connected component $\mE$ of $(d_0,d_1,d_2)$ in $\mG_2$. In the argument what follows, we assume that all of the entries of $(d_0,d_1,d_2)\in G_2$ are in fact in $C^\nu$. 

Let $s$ be any ordinal less than $\nu$. Let us assume that $(d_0,d_1,d_2)\to (e_0,e_1,e_2)$ in $\mG_2$. Then we have that $e_0,e_1,e_2\in C^\nu$. Clearly, the map  given by $a_i\mapsto d_i(s)$ and $b_i\mapsto e_i(s)$, $0\leq i\leq 2$, is a homomorphism from $\mD$ to $\mC$. Then as one checks easily, for each $s<\nu$, either all of the $d_i(s)$ and $ e_i(s)$ coincide where $0\leq i\leq 2$, or the $d_i(s)$ are pairwise different. In the latter case, either $e_i(s)=d_i(s)$ for all $0\leq i\leq 2$ or $e_i(s)=d_i(s)+1$ for all $0\leq i\leq 2$.

By the preceding paragraph,  $\mE$ consists of all 3-tuples $(e_0,e_1,e_2)\in G_2$ such that for each $s<\nu$, $e_0(s)=e_1(s)=e_2(s)=d_0(s)$ if $d_0(s)=d_1(s)=d_2(s)$ and $(e_0(s),e_1(s),e_2(s))\in \{(0,1,2),(2,0,1),(1,2,0)\}$ otherwise.  The edges in $\mE$ are given by $(e_0,e_1,e_2)\to(f_0,f_1,f_2)$  iff for each $s<\nu$, either   $e_i(s)=f_i(s)$ for all $0\leq i\leq 2$ or   $f_i(s)=e_i(s)+1$ for all $0\leq i\leq 2$. Then $\mE$ is a product of  one-element reflexive digraphs and  copies of the subdigraph of $\mC^3$ spanned by $\{(0,1,2),(2,0,1),(1,2,0)\}$ that is isomorphic to $\mC$. Hence $\mE$ is isomorphic to $\mC^{\nu'}$ for some cardinal $\nu'\leq \nu$. Clearly, for any cardinal $\nu'\leq \nu$,  if we set $(d_0(s),d_1(s),d_2(s))=(0,1,2)$ for the coordinates $s<\nu'$ and  $(d_0(s),d_1(s),d_2(s))=(0,0,0)$ for the remaining coordinates $\nu'\leq s<\nu$, then the component containing $(d_0,d_1,d_2)$ in $\mG_2$ is isomorphic to $\mC^{\nu'}$.  By taking into account that $\mG_1$ has a component isomorphic to $\mC^\kappa$, $\mG_2$ has a component isomorphic to $\mC^{\lambda}$ for any cardinal $\lambda\leq\kappa$ and $\mG_2$ is a disjoint union of such components. Notice that by our construction, any component of $\mG_2$ has at most $2^\kappa$ isomorphic copies in $\mG_2$. So $\mG_2$ has at most $\kappa 2^\kappa=2^\kappa$ components. Since each component of $\mG_2$ has at most $2^\kappa$ elements, $\mG_2$ is a compatible digraph of cardinality $2^\kappa$ in $\var V$.

Let $\mL$ denote the digraph with $2^\kappa$-many vertices whose edge relation is the equality relation. Since $\mG_2$ is a compatible digraph of cardinality $2^\kappa$ in $\var V$, $\mL$ is also a compatible digraph in $\var V$. Finally, let $\mG_3=\mG_2\times \mL$. Since both $\mG_2$ and $ \mL$ are compatible in $\var V$, so is $\mG_3$. Clearly, $\mG_3$ is a disjoint union of the powers $\mC^\lambda$ when $\lambda$ runs through the cardinals at most $\kappa$ where each copy of $\mC^\lambda$ appears precisely $2^\kappa$-many times as a component of $\mG_3$.
\end{proof}

Now by the use of the preceding corollary, we can easily prove the main result of the paper.

\begin{theorem}
The filter of the interpretability types of all Taylor varieties is prime in $\LL$.
\end{theorem}
\begin{proof}
It suffices to prove that the join of any two non-Taylor varieties is non-Taylor. Let $\var V$ and $\var W$ be arbitrary non-Taylor varieties. By the preceding corollary, there exists an infinite cardinal $\kappa$ such that the digraph $\dot\bigcup_{\lambda\leq \kappa}2^\kappa\mC^{\lambda}$ is a compatible digraph in both varieties $\var V$ and $\var W$. Therefore, this digraph is also compatible in $\var V\vee\var W$. Then by using the corollary again, $\var V\vee\var W$ is a non-Taylor variety.
\end{proof}

Finally, we would like to compare the interpretability types of the non-Taylor varieties that are defined via the digraphs occurring in the statement of Corollary~\ref{kappa}. For an infinite cardinal $\kappa$, let $\var V_\kappa$ denote the variety generated by the algebra $\alg A_\kappa$ whose universe is $A_\kappa=\dot\bigcup_{\lambda\leq \kappa}2^{\kappa} C^{\lambda}$ where $\lambda$ is a cardinal, and whose basic operations are the polymorphisms of $\mA_\kappa=\dot\bigcup_{\lambda\leq \kappa}2^{\kappa}\mC^{\lambda}$. The {\em beth numbers} are the cardinals defined by the transfinite recursion $\beth_0\coloneqq\aleph_0$ and for any non-zero ordinal $\alpha$, $\beth_{\alpha}\coloneqq\sup_{\beta<\alpha}2^{\beth_{\beta}}$ where $\beta$ is an ordinal.

\begin{proposition} The following hold. 
\begin{enumerate}
\item For any two infinite cardinals $\kappa_1$ and $\kappa_2$ with $2^{\kappa_1}<2^{\kappa_2}$, $\var V_{\kappa_1}$ interprets in $\var V_{\kappa_2}$, but $\var V_{\kappa_2}$ does not interpret in $\var V_{\kappa_1}$.
\item The interpretability types of the non-Taylor varieties $\var V_{\beth_{\alpha}}$ where $\alpha$  is an ordinal form a chain of unbounded size in $\LL$.
\end{enumerate}
\end{proposition}

\begin{proof}
 Item (1) implies item (2) obviously. So we prove that (1) holds. Within the proof, we use  the abbreviations $\mA_i=\mA_{\kappa_i}$, $\alg A_i=\alg A_{\kappa_i}$, and $\var V_i=\var V_{\kappa_i}$ for $1\leq i\leq2$. Since 
$$\mA_1^{\kappa_2}=(\dot\bigcup_{\lambda\leq \kappa_{1}}2^{\kappa_{1}}\mC^{\lambda})^{\kappa_{2}}\cong \dot\bigcup_{\lambda\leq \kappa_{2}}2^{\kappa_{2}}\mC^{\lambda}=\mA_2,$$
$\mA_2$ is a compatible digraph in $\var V_1$. Therefore, $\var V_1$ interprets in $\var V_2$.  We want to prove that $\var V_2$ does not interpret in $\var V_1$. To the contrary, let us assume that $\zeta$ is a clone homomorphism from the clone of $\var V_2$ to the clone $\var V_1$. Clearly, the unary constant operations in the clone of $\var V_2$ are mapped to unary constant operations in the clone of $\var V_1$ by $\zeta$. Since $\mA_2$ has $2^{\kappa_2}$ components, $\mA_1$ has $2^{\kappa_1}$ elements, and $2^{\kappa_1}<2^{\kappa_2}$,  there exist two elements $a$ and $b$ in two different components of $\mA_2$ such that $\zeta(\hat a)=\zeta(\hat b)$ where $\hat a$ and $\hat b$ denote the unary constant term operations corresponding to $a$ and $b$ in the clone of $\var V_2$, respectively.  Clearly, the ternary operation $f(z,x,y)$ defined to be $x$ if $z$ is in the component containing $a$, and $y$ otherwise is a basic operation of $\alg A_2$. Then the identities $f(\hat a,x,y)=x,\ f(\hat b,x,y)=y$ hold in the clone of $\var V_2$. By applying $\zeta$ and using that $\zeta(\hat a)=\zeta(\hat b)$, the identities $$x=\zeta(f)(\zeta(\hat a),x,y)=\zeta(f)(\zeta(\hat b),x,y)=y$$ must be satisfied by the clone of $\var V_1$. So in particular, $\alg A_1\in \var V_1$ must have one element, a contradiction. 
\end{proof}

In particular, it is not possible to characterize non-Taylor varieties by having a fixed compatible digraph of the form
$\dot\bigcup_{t\in T}\mC^{\kappa_t}$ where $T$ is a non-empty set and the $\kappa_t$ are cardinals, not all zero. Indeed, if $\dot\bigcup_{t\in T}\mC^{\kappa_t}$ was compatible  in all non-Taylor varieties, then  by Corollary \ref{kappa}, there would be a $\kappa$ such that  $\mA_\kappa$ would also be compatible, but by the first part of the preceding proposition $\mA_\kappa$ is not compatible in the non-Taylor variety $\var V_{2^\kappa}$, a contradiction. Notice that this behavior of non-Taylor varieties sharply differs from that of the idempotent non-Taylor varieties. Indeed, the latter varieties are characterized by having $\mC$ as a compatible digraph.

\section*{Acknowledgement}
We would like to thank the anonymous referee for their detailed comments and suggestions that led to an improved version of our manuscript.



\begin{thebibliography}{}
\bibitem{GT} Garcia O. C and Taylor, W; The lattice of interpretability types of varieties, Mem. Amer. Math. Soc. 50 (1984) v+125.
\bibitem {GMZ1} Gyenizse, G, Mar\'oti, M and Z\'adori, L; $n$-permutability is not join-prime for $n \geq 5$,
Internat. J. Algebra Comput. 30/8, 1717--1737 (2020).
\bibitem {GMZ2} Gyenizse, G, Mar\'oti, M and Z\'adori, L; Congruence permutability is prime, Proc. Amer. Math. Soc. 150/6, 2733--2739 (2022). 
\bibitem{HM} Hobby, D and McKenzie, R; The structure of finite algebras, Contemporary Mathematics 76, American Mathematical Society, Providence, RI, 1988.
\bibitem{KK} Kearnes, K. A and  Kiss, E. W; The shape of congruence lattices, Mem. Amer. Math. Soc. 222 (2013) viii+169.
\bibitem{KT} Kearnes, K. A and Tschantz, S. T; Automorphism groups of squares and of free algebras, Internat. J. Algebra Comput. 17/3, 461-505 (2007).
\bibitem{LT} Larose, B and Tesson, P; Universal algebra and hardness results for constraint satisfaction problems, Theor. Comput. Sci. 410/18, 1629--1647 (2009).
\bibitem{M} McKenzie, R; Cardinal multiplication of structures with a reflexive relation, Fund. Math. 70, 59--101, (1971).
\bibitem{O} Ol\v s\'ak, M; The weakest nontrivial idempotent equations, Bul. Lond. Math. Soc. 49/6, 1028--1047 (2017).
\bibitem{OP} Opr\v{s}al, J; Taylor's modularity conjecture and related problems for idempotent varieties, Order 35/3, 433--460 (2018).
\bibitem{T} Taylor, W; Varieties obeying homotopy laws, Canad. J. Math., 29/3, 498--527 (1977).
\bibitem{VW} Valeriote, M and Willard, R; Idempotent $n$-permutable varieties, Bulletin of the London Mathematical Society 46,870--880,(2014). 

\end{thebibliography}
\end{document}